\documentclass[11pt]{amsart}

\usepackage{subfiles}

\usepackage{tabu}
\usepackage[margin=1in]{geometry}
\usepackage[dvipsnames]{xcolor}   		             		
\usepackage{graphicx}			
\usepackage{amssymb}
\usepackage{mathrsfs}
\usepackage{amsthm}
\usepackage{amsmath}
\usepackage{stmaryrd}
\usepackage{tikz}
\usepackage{tikz-cd}
\usepackage{accents}
\usepackage{upgreek}
\usepackage{enumerate}
\usepackage{bm}
\usepackage{mathtools}
\usepackage[all]{xy}
\usepackage{caption}

\usepackage{url}
\usepackage{float}
\usepackage{todonotes}
\usepackage{bbm}
\usepackage[normalem]{ulem}

\usepackage[cal=cm]{mathalfa}
\usepackage{xparse}
\usepackage{comment}
\usepackage{cite}

\tikzset{
  commutative diagrams/.cd, 
  arrow style=tikz, 
  diagrams={>=stealth}
}

\theoremstyle{definition}
\newtheorem{innercustomthm}{Theorem}
\newenvironment{customthm}[1]
  {\renewcommand\theinnercustomthm{#1}\innercustomthm}
  {\endinnercustomthm}
  
\makeatletter
\def\@tocline#1#2#3#4#5#6#7{\relax
  \ifnum #1>\c@tocdepth % then omit
  \else
    \par \addpenalty\@secpenalty\addvspace{#2}%
    \begingroup \hyphenpenalty\@M
    \@ifempty{#4}{%
      \@tempdima\csname r@tocindent\number#1\endcsname\relax
    }{%
      \@tempdima#4\relax
    }%
    \parindent\z@ \leftskip#3\relax \advance\leftskip\@tempdima\relax
    \rightskip\@pnumwidth plus4em \parfillskip-\@pnumwidth
    #5\leavevmode\hskip-\@tempdima
      \ifcase #1
       \or\or \hskip 1em \or \hskip 2em \else \hskip 3em \fi%
      #6\nobreak\relax
    \dotfill\hbox to\@pnumwidth{\@tocpagenum{#7}}\par
    \nobreak
    \endgroup
  \fi}
\makeatother

\usetikzlibrary{calc}
\usetikzlibrary{fadings}
\usetikzlibrary{decorations.pathmorphing}
\usetikzlibrary{decorations.pathreplacing}

\newcounter{marginnote}
\DeclareMathAlphabet{\mathpzc}{OT1}{pzc}{m}{it}

\usepackage[backref=page]{hyperref}
\hypersetup{
  colorlinks   = true,          %Colors links instead of ugly boxes
  urlcolor     = blue,          %Color for external hyperlinks
  linkcolor    = purple,          %Color of internal links
  citecolor   = blue             %Color of citations
}

\theoremstyle{definition}
\newtheorem{theorem}{Theorem}[section]

\newtheorem{corollary}[theorem]{Corollary}
\newtheorem{lemma}[theorem]{Lemma}

\newtheorem{remark}[theorem]{Remark}

\newtheorem*{runningexample*}{Running example}

\newtheorem*{aside*}{Aside}

\newtheorem{proposition-definition}[theorem]{Proposition-Definition}

\DeclareMathOperator{\ev}{ev}
\DeclareMathOperator{\vdim}{vdim}

\newcommand{\qu}{/\kern-.7ex/}
\newcommand{\lqu}{\backslash \kern-.7ex \backslash}%left
\newcommand{\on}{\operatorname}

\newcommand{\Tan}{\mathrm{T}}

\newcommand{\EE}{\mathbf{E}}

\newcommand{\st}{\star}

\newcommand{\op}[1]{\operatorname{#1}}
\newcommand{\ol}[1]{\overline{#1}}

\newcommand{\Pcal}{\mathcal{P}}
\newcommand{\Ycal}{\mathcal{Y}}
\newcommand{\Zcal}{\mathcal{Z}}

\newcommand{\bcd}{\begin{center}\begin{tikzcd}}
\newcommand{\ecd}{\end{tikzcd}\end{center}}

\newcommand{\e}{\mathrm{e}}
\newcommand{\Aaff}{\mathbb{A}}
\newcommand{\PP}{\mathbb{P}}
\newcommand{\OO}{\mathcal{O}}
\newcommand{\virt}{\op{virt}}
\newcommand{\HH}{H}

\newcommand{\Norm}{\mathrm{N}}

\newcommand{\Dcal}{\mathcal{D}}
\newcommand{\Xcal}{\mathcal{X}}
\newcommand{\Xfrak}{\mathfrak{X}}
\newcommand{\Ecal}{\mathcal{E}}
\newcommand{\Ccal}{\mathcal{C}}
\newcommand{\Mfrak}{\mathfrak{M}}

\newcommand{\Rder}{\mathrm{R}}

\newcommand{\Kup}{\mathsf{K}}
\newcommand{\Fup}{\mathsf{F}}
\newcommand{\Nup}{\mathsf{N}}
\newcommand{\Gup}{\mathsf{G}}

\NewDocumentCommand{\compatibilitydatum}{m m m m m m O{} O{} O{}}{
\begin{equation*} \begin{tikzcd}[ampersand replacement=\&]
  \: \arrow{r} \& {#1} \arrow{r} \arrow{d}{#7} \& {#2} \arrow{r} \arrow{d}{#8} \& {#3} \arrow{r}{[1]} \arrow{d}{#9} \& \: \\
  \: \arrow{r} \& {#4} \arrow{r} \& {#5} \arrow{r} \& {#6} \arrow{r} \& \:
\end{tikzcd} \end{equation*}}

\NewDocumentCommand{\commutingsquare}{m m m m o O{} O{} O{} O{}}{
\begin{equation}\begin{tikzcd}[ampersand replacement=\&] \label{#5}
  #1 \arrow{r}{#6} \arrow{d}{#7} \& #2 \arrow{d}{#8} \\
  #3 \arrow{r}{#9} \& #4
\end{tikzcd}\IfValueTF{#5}{\label{#5}}{} \end{equation}}

\NewDocumentCommand{\cartesiansquare}{m m m m O{} O{} O{} O{}}{
\begin{equation*}\begin{tikzcd}[ampersand replacement=\&]
  #1 \arrow{r}{#5} \arrow{d}{#6} \arrow[dr, phantom, "\square"] \& #2 \arrow{d}{#7} \\
  #3 \arrow{r}{#8} \& #4
\end{tikzcd} \end{equation*}}

\NewDocumentCommand{\cartesiansquarelabel}{m m m m m O{} O{} O{} O{}}{
\begin{tikzcd}[ampersand replacement=\&]
  #1 \arrow{r}{#6} \arrow{d}{#7} \arrow[dr, phantom, "\square"] \& #2 \arrow{d}{#8} \\
  #3 \arrow{r}{#9} \& #4
\end{tikzcd}\IfValueTF{#5}{\label{#5}}{}
}

\NewDocumentCommand{\triangleofspaces}{m m m O{} O{} O{}}{
\begin{tikzcd} [ampersand replacement=\&]
#1 \arrow{r}{#4} \arrow[bend right]{rr}{#5} \& #2 \arrow{r}{#6} \& #3
\end{tikzcd}}

\title[Local-orbifold correspondence for snc pairs]{The local-orbifold correspondence \\for simple normal crossings pairs}

\author{Luca Battistella, Navid Nabijou, Hsian-Hua Tseng and Fenglong You}


\begin{document}
%\date{\today}

\begin{abstract} For $X$ a smooth projective variety and $D=D_1+\ldots+D_n$ a simple normal crossings divisor, we establish a precise cycle-level correspondence between the genus zero local Gromov--Witten theory of the bundle $\oplus_{i=1}^n \OO_X(-D_i)$ and the maximal contact Gromov--Witten theory of the multi-root stack $X_{D,\vec r}$. The proof is an implementation of the rank reduction strategy. We use this point of view to clarify the relationship between logarithmic and orbifold invariants.\end{abstract}

\maketitle

\tableofcontents

\section*{Introduction} 

\noindent Let $X$ be a smooth projective variety and $D=D_1+\ldots+D_n$ a simple normal crossings divisor with nef components $D_i$. We study the relationship between the genus zero local Gromov--Witten theory of $\oplus_{i=1}^n \mathcal O_X(-D_i)$ and the genus zero orbifold Gromov--Witten theory of the multi-root stack $X_{D,\vec r}$. Our main result is a positive answer to \cite[Conjecture 1.8]{TY20b}:\medskip 

\begin{customthm}{A}[Theorem~\ref{thm: local orbifold}] \label{thm:main} 
Let $\beta$ be a curve class on $X$ with $d_i:=D_i\cdot \beta>0$ for $i\in \{1,\ldots, n\}$. For $r_i$ pairwise coprime and sufficiently large, the following identity holds on the moduli space $\Kup_{0,m}(X,\beta)$ of stable maps to $X$
\begin{align*} \rho_\star [\Kup^{\max}_{0,(I_1,\ldots,I_m)}(X_{D,\vec r},\beta)]^{\virt} = \big( \Pi_{i=1}^n (-1)^{d_i-1} \big) \big( \cup_{j=1}^m \ev_j^\star (\cup_{i \in I_j} D_i) \big) \cap [\Kup_{0,m}(\oplus_{i=1}^n \OO_X(-D_i),\beta)]^{\virt}
\end{align*}
where $I_j\subseteq \{1,\ldots,n\}$ records the set of divisors which the marking $x_j$ is tangent to (see \S \ref{sec: geometric setup} for details), and $\rho$ is the morphism forgetting the orbifold structures.
\end{customthm}\medskip
\noindent This generalises the smooth divisor local-logarithmic correspondence \cite{vGGR} to the simple normal crossings setting, by interpreting the orbifold theory of the multi-root stack as an alternative to the logarithmic theory \cite{TY20c}.

When $D$ is smooth, Theorem~\ref{thm:main} follows from previous results equating both local and orbifold invariants with relative invariants \cite{ACW, vGGR, TY18}. For general $D$, the key observation is that both the local and orbifold theories satisfy a product formula over the space of stable maps to $X$. Theorem~\ref{thm:main} follows immediately, by bootstrapping from the smooth divisor case. This is another manifestation of the ``rank reduction'' technique in Gromov--Witten theory \cite{AbramovichChenLog,NR19}.\medskip

\subsection*{Logarithmic Gromov--Witten theory} Unlike the local and orbifold theories, the logarithmic theory does not satisfy a naive product formula over the space of stable maps to $X$. This observation was used in \cite{NR19} to produce counterexamples to the local-logarithmic conjecture. The same reasoning shows that the orbifold invariants also differ from the logarithmic invariants (and it is easy to find counterexamples beyond the maximal contact setting). In fact, Corollary~\ref{cor: orbifold equals naive} below equates the orbifold invariants with the so-called \emph{naive invariants}, introduced in \cite[\S 3]{NabijouThesis} and studied in \cite{NR19}:

\begin{customthm}{B}[Corollary~\ref{cor: orbifold equals naive}] \label{thm: orbifold equals naive introduction} The orbifold invariants of the multi-root stack coincide with the naive invariants, and hence differ from the logarithmic invariants. This holds for arbitrary choices of contact orders.\end{customthm}

In summary, there are four genus zero maximal contact theories associated to a simple normal crossings pair: logarithmic, orbifold, naive and local. They are related as follows:

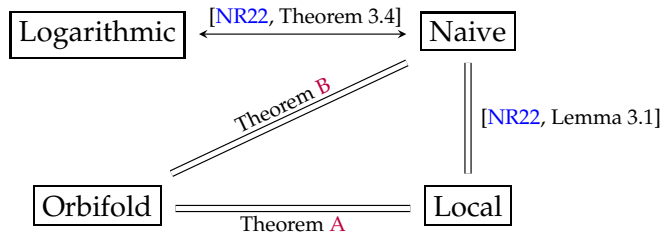
\begin{center}
\hspace*{3em}
\begin{tikzcd}
\fbox{Logarithmic} \ar[rrr,leftrightarrow,"\text{\cite[Theorem~3.4]{NR19}}"] & & & \fbox{Naive} \ar[dd,equal,"\text{\cite[Lemma~3.1]{NR19}}"] \ar[llldd,equal,"\text{Theorem~\ref{thm: orbifold equals naive introduction}}" above,sloped] \\
& & & \\
\fbox{Orbifold} \ar[rrr,equal,"\text{Theorem~\ref{thm:main}}" below] & & & \fbox{Local}
\end{tikzcd}
\end{center}

\noindent The orbifold, local and naive theories all coincide up to combinatorial factors. The logarithmic theory differs in a more essential way, though there is an in-principle procedure which relates it to the other three. 

Despite the failure of the cycle-level local-logarithmic correspondence, there are many choices of targets and insertions for which the correspondence does hold on the numerical level. This occurs when the insertions kill the correction terms described in \cite[Theorem 3.4]{NR19}. In \cite{BBvG,BBvG2,BBvG3} numerous instances of the numerical local-logarithmic correspondence are established: for toric varieties, log Calabi--Yau surfaces and orbifold log Calabi--Yau surfaces; in \cite[\S 5]{NR19} the numerical correspondence is established for product geometries. As a corollary of Theorem~\ref{thm:main}, all of these logarithmic invariants coincide with the corresponding orbifold invariants.

These different theories are approached using very different techniques. Torus localisation, in various guises, has been applied to compute both orbifold and local invariants \cite{KlemmP08,CCITTwisted,FTYMirror}; logarithmic invariants, on the other hand, are typically calculated using tropical correspondence theorems and scattering diagrams \cite{NS06,GPS}. Depending on the context, one technique may be more effective than another. The above correspondences provide bridges between different techniques, thus increasing the roster of tools available for computations in Gromov--Witten theory as a whole.

\subsection*{Relation to previous work} The smooth divisor case of Theorem~\ref{thm:main} follows by combining the orbifold-logarithmic correspondence \cite{ACW, TY18} with the strong form \cite{FanWuWDVV,TY20b} of the local-logarithmic correspondence \cite{vGGR}. Some cases of Theorem~\ref{thm:main} for normal crossings divisors were numerically verified in \cite[\S5.2]{TY20b}, by computing the $J$-functions of both sides.

\subsection*{User's guide} We provide two approaches to rank reduction. The first (\S\ref{sec: rank reduction projection formula}) uses the iterative construction of root stacks and the projection formula, and relies on a local-orbifold correspondence for certain smooth orbifold pairs (Theorem~\ref{thm: local-orbifold for orbifold}). The second (\S\ref{sec: rank reduction product}) uses a product formula for orbifold invariants over the space of stable maps to the coarse moduli space (Theorem~\ref{thm: product formula}). This holds for arbitrary tangency orders but requires a positivity assumption. The identification of orbifold and naive invariants (Corollary~\ref{cor: orbifold equals naive}) is an immediate consequence.

\subsection*{Acknowledgements} We are grateful to Dhruv Ranganathan for conversations on related ideas, and comments on a draft manuscript. We also thank Michel van Garrel and Leo Herr for useful discussions. Finally we thank the referee for a thorough reading of the paper and many helpful comments.

\subsection*{Funding} L.B. is supported by the Deutsche Forschungsgemeinschaft (DFG, German Research Foundation) under Germany’s Excellence Strategy EXC-2181/1 - 390900948 (the Heidelberg STRUCTURES Cluster of Excellence). N.N. is supported by the Herchel Smith Fund. H.-H. T. is supported in part by Simons Foundation Collaboration Grant. F.Y. is supported by the EPSRC grant EP/R013349/1.

\section{Rank reduction I: projection formula}\label{sec: rank reduction projection formula}
\subsection{Geometric setup}\label{sec: geometric setup} Fix a smooth projective variety $X$ and a simple normal crossings divisor $D=D_1+\ldots+D_n \subseteq X$. For a tuple of pairwise coprime and sufficiently large integers $\vec r = (r_1,\ldots,r_n)$, we form the associated multi-root stack:
\begin{equation*} \Xcal = X_{D,\vec r}. \end{equation*}
Consider $m$ marked points $x_1,\ldots,x_m$ and fix an ordered partition of the index set $\{1,\ldots,n\}$ into disjoint subsets $I_1,\ldots,I_m$ such that $\cap_{i \in I_j} D_i$ is nonempty for each $j \in \{1,\ldots,m\}$. Fix a curve class $\beta \in \HH_2^+(X)$ such that $d_i := D_i \cdot \beta > 0$ for all $i$.

We consider a moduli problem of genus zero stable maps relative to $(X,D)$, such that the marking $x_j$ has maximal contact order $d_i$ to each divisor $D_i$ with $i \in I_j$. Notice that some of the $I_j$ may be empty, corresponding to markings with no tangency conditions. Some markings may have positive contact orders along several divisors simultaneously, which implies in particular that they should be mapped to the intersection.

This moduli problem determines associated discrete data for a moduli problem of orbifold stable maps to the multi-root stack $\Xcal$, by taking each marking $x_j$ to have twisting index:
\begin{equation*}s_j=\prod_{i \in I_j} r_i.\end{equation*}
The twisted sector insertion in
\begin{equation*} \mu_{s_j} = \prod_{i \in I_j} \mu_{r_i}\end{equation*}
coincides with the tuple of tangency orders, since the twisting indices on source and target are the same \cite[\S 2.1]{CadmanChen}. We denote the associated moduli space by
\begin{equation*} \Kup^{\max}_{0,(I_1,\ldots,I_m)}(\Xcal,\beta)\end{equation*}
and let $\rho$ denote the morphism which forgets the orbifold structures:
\begin{equation*} \rho \colon \Kup^{\max}_{0,(I_1,\ldots,I_m)}(\Xcal,\beta) \to \Kup_{0,m}(X,\beta).\end{equation*}

\subsection{Local-orbifold correspondence} 
Our main result is a cycle-level correspondence between the multi-root orbifold theory and the local theory of the associated split vector bundle, proving \cite[Conjecture 1.8]{TY20b}:
\begin{theorem}\label{thm: local orbifold} For $r_i$ sufficiently large we have:
\begin{equation*} \rho_\star [\Kup^{\max}_{0,(I_1,\ldots,I_m)}(\Xcal,\beta)]^{\virt} = \big( \Pi_{i=1}^n (-1)^{d_i-1} \big) \left( \cup_{j=1}^m \ev_j^\star (\cup_{i \in I_j} D_i) \right) \cap [\Kup_{0,m}(\oplus_{i=1}^n \OO_X(-D_i),\beta)]^{\virt}.
\end{equation*}\end{theorem}
The case $n=1$ follows by combining the smooth divisor local-logarithmic correspondence \cite{vGGR} in its strong form (see \cite[Introduction]{FanWuWDVV} or \cite[Equation (2)]{TY20b}), together with the smooth divisor logarithmic-orbifold correspondence \cite{ACW,TY18}.

\begin{proof} We proceed by induction on $n$. The base case $n=1$ is discussed above. For the induction step, consider the root stack
\begin{equation*}\Zcal = X_{(D_1,\ldots,D_{n-1}),(r_1,\ldots,r_{n-1})}.\end{equation*}
Letting $p \colon \Zcal \to X$ be the morphism to the coarse moduli space and $\Dcal_n=p^{-1} D_n$, we have:
\begin{equation*} \Xcal = \Zcal_{\Dcal_n,r_n}.\end{equation*}
The ordered partition $(I_1,\ldots,I_m)$ of $\{1,\ldots,n\}$ induces a partition $(J_1,\ldots,J_m)$ of $\{1,\ldots,n-1\}$ by setting $J_j = I_j \setminus \{ n\}$. Consider the tower of moduli spaces:
\bcd
\Kup^{\max}_{0,(I_1,\ldots,I_m)}(\Xcal,\beta) \ar[r,"\psi"] \ar[rr, bend right=20, "\rho"] & \Kup_{0,(J_1,\ldots,J_m)}^{\max}(\Zcal,\beta) \ar[r,"\varphi"] & \Kup_{0,m}(X,\beta).
\ecd
The induction hypothesis gives
\begin{equation}\label{eqn: induction hypothesis bootstrap} 
	\varphi_\star [\Kup^{\max}_{0,(J_1,\ldots,J_m)}(\Zcal,\beta)]^{\virt} = \big( \Pi_{i=1}^{n-1} (-1)^{d_i-1} \big) \left( \cup_{j=1}^m \ev_j^\star (\cup_{i \in J_j} D_i) \right) \cap [\Kup_{0,m}(\oplus_{i=1}^{n-1} \OO_X(-D_i),\beta)]^{\virt}
\end{equation}
while Theorem~\ref{thm: local-orbifold for orbifold} below establishes a local-orbifold correspondence for the smooth orbifold pair $(\Zcal,\Dcal_n)$, giving
\begin{align} \psi_{\star} [\Kup^{\max}_{0,(I_1,\ldots,I_m)}(\Xcal,\beta)]^{\virt} & = (-1)^{d_n-1} \ev_{j_n}^\star \Dcal_n \cap [\Kup_{0,(J_1,\ldots,J_m)}^{\max}(\OO_{\Zcal}(-\Dcal_n),\beta)]^{\virt} \nonumber \\
& \label{eqn: local=orbifold for multi-root} = (-1)^{d_n-1} \ev_{j_n}^\star \Dcal_n \cdot \e(\Rder^1 \pi_\star f^\star \OO_{\Zcal}(-\Dcal_n)) \cap [\Kup_{0,(J_1,\ldots,J_m)}^{\max}(\Zcal,\beta)]^{\virt} \end{align}
where $j_n \in \{1,\ldots,m\}$ is the unique index such that $n \in I_{j_n}$. Since $\Dcal_n = p^\star D_n$ is pulled back from $X$ we have
\begin{align*} \ev_{j_n}^\star \Dcal_n  & = \varphi^\star \ev_{j_n}^\star D_n \\
	 \e(\Rder^1 \pi_\star f^\star \OO_{\Zcal}(-\Dcal_n)) & = \varphi^\star \e(\Rder^1 \pi_\star f^\star \OO_X(-D_n))
\end{align*}
The latter equation follows from the projection formula and the following fact. The pullback $\varphi^\star C$ of the universal curve on $\Kup_{0,m}(X,\beta)$ coincides with the coarsening of the universal curve
\[ \mathcal C \to \Kup_{0,(J_1,\ldots,J_m)}^{\max}(\Zcal,\beta)\]
which makes the composite universal map $\mathcal C\to\mathcal Z\to X$ representable. Therefore the structure sheaves of the universal curves are preserved by pushforward \cite[Theorem~3.1]{AOV}. The result then follows from \eqref{eqn: induction hypothesis bootstrap} and \eqref{eqn: local=orbifold for multi-root}, the projection formula for $\varphi$ and the splitting of the obstruction bundle for the local theory of $\oplus_{i=1}^n \OO_X(-D_i)$.\end{proof}

\subsection{Local-orbifold correspondence for smooth orbifold pairs} \label{sec: proof of local orbifold for orbifolds} It remains to establish the local-orbifold correspondence for the smooth orbifold pair $(\Zcal,\Dcal_n)$, used in the proof above.

\begin{theorem}\label{thm: local-orbifold for orbifold} With notation as in the proof of Theorem~\ref{thm: local orbifold}, we have:
\begin{equation*}\psi_\star [\Kup^{\max}_{0,(I_1,\ldots,I_m)}(\Xcal,\beta)]^{\virt} = (-1)^{d_n-1} \ev_{j_n}^\star \Dcal_n \cap [\Kup_{0,(J_1,\ldots,J_m)}^{\max}(\OO_{\Zcal}(-\Dcal_n),\beta)]^{\virt}.\end{equation*}\end{theorem}

We establish this result only in the setting we require, namely when $\Zcal$ is a multi-root stack and $\Dcal_n$ is a divisor pulled back from the coarse moduli space. The proof adapts the arguments of \cite{vGGR} but subtleties arise due to the twisted sectors of $\Zcal$, which encode tangencies with respect to the divisors $D_1,\ldots,D_{n-1}$. These complicate a crucial dimension count in the proof (\S \ref{sec: Y-vanishing}) and also affect the final multiplicity calculation (\S \ref{sec: special graph contribution}).

\begin{remark}
	It is unclear whether the correspondence holds in greater generality. If the divisor has generic stabiliser then the dimension count (\S \ref{sec: Y-vanishing}) can fail, so at best the result must be established via other methods. Moreover in this case the multiplicity arising from the contribution of the special graph (\S \ref{sec: special graph contribution}) is different, hinting that any generalisation of the correspondence will in fact require a new formulation.
\end{remark}

\subsubsection{Setting up the degeneration formula} Let $\Xfrak$ be the degeneration to the normal cone of $\Dcal_n \subseteq \Zcal$, and let $M$ be the degeneration to the normal cone of $D_n\subseteq X$. 
\begin{lemma}
 $\Xfrak$ is the root stack of $M$ along the strict (equivalently, total) transforms of the divisors $D_i \times \Aaff^1$ for $i\in \{1,\ldots,n-1\}$.
\end{lemma}
\begin{proof}
 The divisors $D_i \times \Aaff^1$ intersect the blowup centre $D_n\times\{0\}$ transversely, hence the strict and total transforms coincide. Denote these by $T_i \subseteq M$. Each $T_i$ admits a root on $\Xfrak$, namely the pullback of the rooted divisor $\Dcal_i \subseteq \Zcal$ along the composition $\Xfrak \to \Zcal \times \Aaff^1 \to \Zcal$. By the universal property of the root stack we obtain a morphism
 \[ \Xfrak \to M_{(T_1,\ldots,T_{n-1}),(r_1,\ldots,r_{n-1})}\] 
and a local calculation shows that this is an isomorphism. 
\end{proof}

The general fibre of the family $\Xfrak \to \Aaff^1$ is:
\begin{equation*} \Zcal = X_{(D_1,\ldots,D_{n-1}),(r_1,\ldots,r_{n-1})}. \end{equation*}
The central fibre consists of two components $\Zcal$ and $\Ycal$ meeting along $\Dcal_n$. Here $\Ycal$ is obtained by rooting the bundle $Y=\PP_{D_n}(\Norm_{D_n|X}\oplus \OO_{D_n})$ along the divisors $\pi^{-1}(D_i \cap D_n)$ for $i\in\{1,\ldots,n-1\}$. There is a cartesian square
\bcd
\Ycal \ar[r] \ar[d,"\pi" left] \ar[rd,phantom,"\square"] & Y \ar[d,"\pi"] \\
\Dcal_n \ar[r] & D_n
\ecd
where we note that $\Dcal_n$ is itself a multi-root stack along a simple normal crossings divisor
\begin{equation*} \Dcal_n = (D_n)_{(E_1,\ldots,E_{n-1}),(r_1,\ldots,r_{n-1})}\end{equation*}
where $E_i = D_i \cap D_n \subseteq D_n$. We let $\Ecal_i = E_i/r_i \subseteq \Dcal_n$ be the corresponding gerby divisor.

Each connected component of the rigidified inertia stack $\ol{\mathcal{I}}(\Dcal_n)$ is a rigidification of a closed stratum $\bigcap_{i\in I}\Ecal_i$ of the divisor $\Ecal_1+\ldots+\Ecal_{n-1} \subseteq \Dcal_n$ (including the stratum $\Dcal_n$ corresponding to the empty intersection). This rigidification is obtained from $\bigcap_{i\in I} E_i$ by rooting along the intersection with $E_j$ for all $j\notin I$. This description of the twisted sectors is crucial for understanding the structure of the degeneration formula below.

Finally, $\Dcal_0\subseteq\Ycal$ will denote the section of the bundle consisting of its intersection with $\Zcal$, while $\Dcal_\infty\subseteq \Ycal$ will denote the intersection of the central fibre of $\Xfrak$ with the strict transform $\mathfrak{D}$ of $\Dcal_n \times \Aaff^1$.

Consider $\mathfrak L=\on{Tot} \mathcal O_{\mathfrak X}(-\mathfrak D)$. This forms a family of (non-proper) targets over $\Aaff^1$. The general fibre is $\on{Tot} \mathcal O_{\mathcal Z}(-\Dcal_n)$ and the central fibre is a union of $\Zcal \times \mathbb A^1$ and $\on{Tot} \mathcal O_{\mathcal Y}(- \Dcal_\infty)$.

We apply the degeneration formula \cite{AF11} to $\mathfrak{L}$. The components of the central fibre are indexed by bipartite graphs $\Gamma$. The vertices $v \in \Gamma$ are partitioned into $\Zcal$-vertices and $\Ycal$-vertices, and the associated moduli spaces $\Kup_v$ are spaces of maps to expansions of the rooted pairs
\begin{equation*} (\Zcal\times \Aaff^1,\Dcal_n\times\Aaff^1) \qquad \text{and} \qquad (\OO_{\Ycal}(-\Dcal_\infty),\Dcal_0\times\Aaff^1) \end{equation*}
respectively, as defined in \cite[\S 3]{AF11}. Working with expansions is inconsequential, since the pushforward of the virtual class matches that of the space of maps to the corresponding root stack without expansions \cite[Theorem 2.2]{ACW}. We denote the twisting index by $r_n$. In the original formulation \cite[\S 3.4]{AF11}, $r_n$ is required to be divisible by all contact orders at the gluing nodes, but by \cite{TY18} this condition can be removed without affecting the invariants. We assume therefore that $r_n$ is large and coprime to each of $r_1,\ldots,r_{n-1}$.

The component $\Kup_{\Gamma}$ associated to $\Gamma$ maps to the fibre product
\begin{equation} \label{diag: degeneration formula gluing}
\begin{tikzcd}
\Kup_{\Gamma} \ar[r,"\Phi"] & \Fup_\Gamma \ar[r] \ar[d] \ar[rd,phantom,"\square"] & \prod_v \Kup_v \ar[d] \\
\, & \prod_e \ol{\mathcal{I}}(\Dcal_n\times \Aaff^1) \ar[r,"\Delta"] & \prod_e \ol{\mathcal{I}}(\Dcal_n\times \Aaff^1)^2	
\end{tikzcd}
\end{equation}
with respect to the evaluation maps to the rigidified inertia stack of the join divisor. The virtual class of $\Kup_\Gamma$ pushes forward to a multiple of the virtual class with which $\Fup_\Gamma$ is endowed by virtue of this diagram; the virtual degree of the morphism $\Phi$ is well-understood \cite[Proposition 5.9.1]{AF11}. Each space $\Kup_v$ decomposes as a disjoint union of substacks obtained as preimages of connected components of the inertia stack. 

After pushing forward to the space of stable maps to $\Zcal$, the degeneration formula gives an equality of classes
\begin{equation} \label{eqn: degeneration formula}
[\Kup^{\max}_{0,(J_1,\ldots,J_m)}(\OO_{\Zcal}(-\Dcal_n),\beta)]^{\on{virt}}=\sum_{\Gamma}\dfrac{1}{\lvert E(\Gamma)\rvert !} \cdot \Psi_\star [\Kup_{\Gamma}]^{\virt} \end{equation}
where $\Psi$ is the composition:
\begin{equation*} \Kup_{\Gamma} \to \Kup(\mathfrak{L}_0) \to \Kup(\Zcal).\end{equation*}
Let $j=j_n$ be the index of the marking at which we wish to impose tangency to $D_n$ (as in the proof of Theorem~\ref{thm: local orbifold}) and cap both sides of \eqref{eqn: degeneration formula} with $\ev_j^\star \mathfrak{D}$. The left-hand side gives the local invariants of $\OO_{\Zcal}(-\Dcal_n)$ capped with $\ev_j^\star \Dcal_n$. Our aim is to show that all but one of the terms on the right-hand side vanish.
	
\subsubsection{First vanishing: $\Zcal$-vertices}
Suppose first that there is a $\Zcal$-vertex $v \in \Gamma$ with $k>1$ adjacent edges. For each adjacent edge $e$ the corresponding evaluation map factors (locally) through a specific component of the rigidified inertia stack. Such a component is obtained by rigidifying a (possibly empty) intersection of the divisors $\Ecal_i$ in $\Dcal_n$. We denote this by $\Ecal_e$. The product of evaluation maps thus takes the form:
\begin{equation*} \Kup_v \to \prod_e (\Ecal_e \times \Aaff^1). \end{equation*}
However, since the source curve is proper, the map to affine space is constant. This ensures that all the point evaluations agree, i.e. the above map factors through the closed substack
\begin{equation*} \bigg(\prod_e \Ecal_e \bigg) \times \Aaff^1 \hookrightarrow \prod_e (\Ecal_e \times \Aaff^1).
\end{equation*}
We now follow the argument of \cite[Lemma 3.1]{vGGR}. There is a cartesian diagram:
\bcd
\Fup_{\Gamma} \ar[r]\ar[d] \ar[rd,phantom,"\square" right=0.49] & \Kup_v \times \prod_{v^\prime}\Kup_{v^\prime}\ar[d]\\
\left(\prod_e \Ecal_e\right) \times \Aaff^1 \times \prod_{e^\prime} \ol{\mathcal{I}}(\Dcal_n\times\Aaff^1) \ar[r,"\widetilde\Delta"]\ar[d,"\iota"]\ar[rd,phantom,"\square"]  & \left( \left(\prod_e \Ecal_e\right) \times \Aaff^1\right)^2 \times \prod_{e^\prime} \ol{\mathcal{I}}(\Dcal_n\times\Aaff^1)^2 \ar[d,"\iota^\prime"] \\
\prod_e (\Ecal_e \times \Aaff^1) \times \prod_{e^\prime} \ol{\mathcal{I}}(\Dcal_n\times\Aaff^1) \ar[r,"\Delta"]  & \prod_e (\Ecal_e \times \Aaff^1)^2 \times \prod_{e^\prime} \ol{\mathcal{I}}(\Dcal_n\times\Aaff^1)^2.
\ecd
The excess intersection formula \cite[Theorem 6.3]{FultonIntersectionTheory} gives 
\[\Delta^!=\operatorname{c}_{k-1}(E)\cap\widetilde\Delta^!\]
where $E$ is the excess bundle, which in this case \cite[Example 6.3.2]{FultonIntersectionTheory} is equal to
\[E=\widetilde\Delta^\star \Norm_{\iota^\prime}/\Norm_{\iota}\]
which is clearly trivial if $k>1$. It follows that $\Delta^!=0$ and so the contribution of $\Gamma$ vanishes.

\subsubsection{Second vanishing: $\Ycal$-vertices}\label{sec: Y-vanishing}
We conclude that the only graphs which can contribute are those with a single $\Ycal$-vertex. Let $v \in \Gamma$ be such a vertex. This corresponds to a space of expanded maps to the rooted pair $(\OO_{\Ycal}(-\Dcal_\infty),\Dcal_0\times \Aaff^1)$. Recall that $\Ycal$ is a projective bundle over the divisor $\Dcal_n$. Suppose that in the discrete data for $\Kup_v$ either:
\begin{itemize}
\item the curve class is not a multiple of the fibre class, or;
\item there are at least three special points.
\end{itemize}
This ensures that the corresponding moduli space of stable maps $\Kup_v(\Dcal_n)$ to the base of the bundle is well-defined. There is a projection
\begin{equation}\label{eqn: projection to D_n mapping space}  \Kup_v \to \Kup_v(\Dcal_n)\end{equation}
and we claim that the virtual class pushes forward to zero along this morphism. Since there is a natural bijection between the twisted sectors in $\Ycal$ and the twisted sectors in $\Dcal_n$, the age contributions to the virtual dimensions coincide. From this one deduces
$$\vdim \Kup_v = \vdim \Kup_v(\Dcal_n) + 2$$
and so the claim holds if we show that \eqref{eqn: projection to D_n mapping space} satisfies the virtual pushforward property \cite[Definition 3.1]{ManolachePush} (we note that this dimension count can fail if $\Dcal_n$ is allowed to have generic stabiliser). By \cite[Theorem 2.2]{ACW}, it is equivalent to show that
\begin{equation*} \Kup_v(\Ycal_{\Dcal_0,r_n}) \to \Kup_v(\Dcal_n) \end{equation*}
satisfies the virtual pushforward property. For this we adapt the arguments of \cite[\S4]{vGGR}. Let:
\begin{equation*}s=\Pi_{i=1}^{n} r_i, \qquad t = \Pi_{i=1}^{n-1} r_i=s/r_n.	
\end{equation*}
By representability, the stabiliser groups of the source curve of any stable map to $\Ycal_{\Dcal_0,r_n}$ (respectively $\Dcal_n$) must have order dividing $s$ (respectively $t$). We denote the monoids of effective curve classes by:
\begin{equation*} A= \HH_2^+(\Ycal_{\Dcal_0,r_n}), \qquad B = \HH_2^+(\Dcal_n).\end{equation*}
Now consider the following diagram, involving moduli stacks of prestable twisted curves with homology weights
\begin{equation} \label{eqn: virtual pushforward diagram}
\begin{tikzcd}
	\Kup_v(\Ycal_{\Dcal_0,r_n}) \ar[r,"\upsilon"] \ar[rd, bend right] & \Gup \ar[r] \ar[d] \ar[rd,phantom,"\square"right] & \Kup_v(\Dcal_n) \ar[d] \\
\, & \Mfrak^{s-\operatorname{tw}}_{A} \ar[r,"\nu"] & \Mfrak^{t-\operatorname{tw}}_{B}
\end{tikzcd}
\end{equation}
in which the morphism $\nu$ contracts unstable curve components with vertical homology class and coarsens the $r_n$--twisting (this morphism is the composition of an \'etale cover followed by a root construction).

From the short exact sequence of relative tangent bundles associated to the smooth projection $\Ycal_{\Dcal_0,r_n} \to \Dcal_n$ we obtain a compatible triple for the triangle in \eqref{eqn: virtual pushforward diagram}. We note that unlike when the target is a variety, we may have
\begin{equation*}\HH^1(\Ccal, f^\star \Tan_{\Ycal_{\Dcal_0,r_n}/\Dcal_n})\neq 0 \end{equation*}
if components of $\Ccal$ are mapped into the rooted divisor. Thus the morphism $\upsilon$ is not typically smooth, but it is always virtually smooth which is sufficient. The arguments given in \cite[Lemma~5.1 and Proposition~5.3]{vGGR} then apply verbatim, showing that the virtual pushforward property holds and that the contribution of $\Gamma$ vanishes.

\subsubsection{Contribution of the special graph}\label{sec: special graph contribution}
We conclude that the only graphs $\Gamma$ which contribute are those with a single $\Ycal$-vertex $v_1$ with at most two special points and curve class a multiple of the fibre class $F$. Since $v_1$ must contain at least one node as well as the marking $x_{j}$ we are left with a single graph $\Gamma$, consisting of:
\begin{itemize}
\item a $\Zcal$-vertex $v_0$ supporting all the markings except $x_{j}$ and with curve class $\beta_0=\beta$;
\item a $\Ycal$-vertex $v_1$ supporting the marking $x_{j}$ and with curve class $\beta_1=d_n \cdot F$ for $d_n=\Dcal_n\cdot \beta$.
\end{itemize}
These are connected along a single edge $e$ and the diagram \eqref{diag: degeneration formula gluing} reduces to the following:
\bcd
\Kup_\Gamma \ar[r] & \Fup_{\Gamma} \ar[r] \ar[d] \ar[rd,phantom,"\square"] & \Kup_{v_0} \times \Kup_{v_1} \ar[d] \\
\, & \ol{\mathcal{I}}(\Dcal_n \times\Aaff^1) \ar[r] & \ol{\mathcal{I}}(\Dcal_n \times\Aaff^1)^2.
\ecd
Recall that the component $\Kup_{\Gamma}$ of the central fibre is virtually finite over the fibre product $\Fup_{\Gamma}$. Let
\begin{equation*} J_j \subseteq \{1,\ldots,n-1\}\end{equation*}
be the subset recording those divisors amongst $D_1,\ldots,D_{n-1}$ which the marking $x_{j}$ is tangent to. This tangency is encoded in twisted sector insertions which are imposed on both the general and central fibres. In $\Kup_{v_1}$ these correspond to age constraints with respect to the bundles:
\begin{equation*}\OO_\Ycal(\pi^{-1}\Ecal_i) =\pi^\star \OO_{\Dcal_n}(\Ecal_i).\end{equation*}
 Since the curve class is a multiple of a fibre, these bundles have zero degree when pulled back to the source curve. It follows from parity considerations that $\Kup_{v_1}$ is empty unless the nodal marking $q$ corresponding to the edge $e$ also carries twisted sector insertions, which are opposite to those at $x_{j}$. This means we must have
\begin{equation*} \operatorname{age}_{q} \pi^\star  \OO_{\Dcal_n}(\Ecal_i)= 1 - \operatorname{age}_{x_j} \pi^\star \OO_{\Dcal_n}(\Ecal_i) \end{equation*}
for all $i \in J_{j}$. By the inversion of the band in the evaluation maps, we then have the opposite ages for the nodal marking $q$ on $\Kup_{v_0}$. It follows that the vertex $v_0$ contributes the orbifold invariants of the root stack $\Zcal_{\Dcal_n,r_n} = \Xcal$ with twisted sector insertions imposing maximal tangency of a single marking $q$ with respect to all divisors $D_i$ for $i \in I_j=J_j \cup \{n\}$, as required.

For the contribution of $v_1$, notice that $\ev_q$ takes values in a component of $\ol{\mathcal{I}}(\Dcal_n)$ which is naturally isomorphic to the rigidification of:
\begin{equation*} \bigcap_{i \in J_j} \Ecal_i. \end{equation*}
We denote this rigidification by $\Ecal_{J_j}$. A direct calculation shows that:
\begin{equation*} \operatorname{vdim} \Kup_{v_1} = \dim \Ecal_{J_j} + 1.\end{equation*}
There is a divisorial insertion $\ev_j^\star \Dcal_\infty$ on $\Kup_{v_1}$ and the contribution of $v_1$ can be expressed as the unique $m \in \mathbb{Q}$ such that:
\begin{equation*} (\ev_q)_\star ( \ev_j^\star \Dcal_\infty \cap [\Kup_{v_1}]^{\virt} )= m\cdot [\Ecal_{J_j}].
\end{equation*}
This can be computed by restricting to the fibre of a general point in $\Ecal_{J_j}$. Working locally around this point, the gerbes $\Ecal_i$ become trivial, so that we obtain a space of maps to:
\begin{equation*} \PP(r_n,1) \times \prod_{i \in J_j} \mathcal{B}\mu_{r_i}.\end{equation*}
The maps to the $\mathcal{B}\mu_{r_i}$ are uniquely determined, and each has an automorphism factor of $1/r_i$. This cancels with the automorphism factor arising from the Chen--Ruan intersection pairing on the inertia stack of the join divisor \cite[\S 5.2.3]{AF11}.

We are left with a computation on $\PP(r_n,1)$. The contribution is a local invariant capped with an insertion of $\ev_{j}^\star (\infty)$. The latter insertion can be factored out via the divisor axiom; the obstruction bundle of the local theory is pulled back along the map forgetting a marking, since the structure sheaves of the universal curves are preserved by pushforward along stabilisation. The remaining local invariant can be computed by localisation. The end result \cite[(21)]{JPTNotes} is
\begin{equation*} (d_n) \left( \dfrac{(-1)^{d_n-1}}{d_n^2}\right) = \dfrac{(-1)^{d_n-1}}{d_n} \end{equation*}
which combines with the gluing factor $d_n$ appearing in the degeneration formula to complete the proof of Theorem~\ref{thm: local-orbifold for orbifold}.\qed

\section{Rank reduction II: relative product formula}\label{sec: rank reduction product}

\noindent Having established the main Theorem~\ref{thm: local orbifold}, we now present an alternative approach, also based on the rank reduction philosophy. While this approach is less general, requiring a positivity assumption, we have chosen to include it since the ``relative product formula'' it employs provides valuable insight into the geometry of maps to the multi-root stack, and clarifies the relationship to logarithmic invariants. Moreover the main result does not require the maximal contact assumption.

\subsection{Convex embeddings} \label{sec: convex embeddings}

\noindent As before, fix a smooth projective variety $X$ and a simple normal crossings divisor $D=D_1+\ldots+D_n \subseteq X$. To ease notation we will assume from now on that $n=2$; the extension to the general case follows by induction.

We will assume throughout this section that there exists a simple normal crossings pair $(P,H=H_1+H_2)$ with $P$ convex, and a closed embedding $X \hookrightarrow P$ such that $D_i = X \cap H_i$ for each $i$. In this situation we call $(X,D)$ a \textbf{convex embedding}. Two important cases encompassed by this definition are:
\begin{enumerate}
\item $X$ convex and $D_i$ arbitrary;
\item $X$ arbitrary and $D_i$ very ample.	
\end{enumerate}
All definitions and proofs will be given first in the case where $X$ itself is convex, and then extended to convex embeddings via virtual pullback.

\subsection{Relative product formula for root stacks} \label{sec: product formula orbifold} As in \S\ref{sec: rank reduction projection formula}, we fix discrete data for a moduli problem of genus zero relative stable maps to $(X,D)$: a curve class $\beta \in \HH_2^+(X)$, a number of marked points $m$, and specified tangency orders to $D_1$ and $D_2$ at the marked points. Note that we do not require the contact orders to be maximal at this point.

Choose large coprime integers $r_1$ and $r_2$ and consider the root stacks:
\begin{equation*} \Xcal_1 = X_{D_1,r_1} \qquad \Xcal_2 = X_{D_2,r_2}. \end{equation*}
These both have $X$ as their coarse moduli space. For each $\Xcal_i$ we can set up data for a moduli space of orbifold stable maps, by taking every marking to have twisting index $r_i$. The twisted sector insertion in $\mu_{r_i}$ coincides with the tangency order, since the twisting indices on source and target are the same \cite[\S 2.1]{CadmanChen}. Consider now the multi-root stack:
\begin{equation*}\Xcal = \Xcal_1 \times_X \Xcal_2.\end{equation*}
Just as before, we may construct discrete data for a space of orbifold stable maps to $\Xcal$. Markings tangent to both $D_1$ and $D_2$ will have twisting index $r_1r_2$, and the twisted sector insertion is the unique element of $\mu_{r_1r_2}$ which maps to the correct pair of tangencies under the canonical isomorphism $\mu_{r_1r_2} = \mu_{r_1} \times \mu_{r_2}$. From now on the discrete data will be suppressed from the notation

In this section we show that the theory of orbifold stable maps satisfies a relative product formula over the space of maps to the coarse moduli space. To be more precise:
\begin{theorem} \label{thm: product formula} There exists a diagram
\begin{equation} \label{eqn: main cartesian square}
\begin{tikzcd}
\Kup(\Xcal) \ar[r,"\nu"] & \mathcal P \ar[r] \ar[d,"\rho"] \ar[rd,phantom,"\Box"] & \Kup(\Xcal_1) \times \Kup(\Xcal_2) \ar[d,"\rho_1\times \rho_2"] \\
& \Kup(X) \ar[r,"\Delta_{\Kup(X)}"] & \Kup(X) \times \Kup(X)
\end{tikzcd}
\end{equation}
such that, when $X$ is convex, we have:
\begin{equation} \label{eqn: product formula orbifold} \nu_*[\Kup(\Xcal)]^{\virt} = \Delta_{\Kup(X)}^!\left( [\Kup(\Xcal_1)]^{\virt} \times [\Kup(\Xcal_2)]^{\virt} \right).\end{equation}
\end{theorem}

\begin{proof} \label{sec: proof of square}
The morphism $\nu \colon \Kup(\Xcal)\to\Pcal$ is obtained by taking relative coarse moduli spaces, see \cite[\S9]{AbramovichVistoli} and \cite[Theorem 3.1]{AOV}. For each $i$ the partial coarsening $\Ccal \to \Ccal_i$ is initial amongst maps $\Ccal \to \Ycal$ through which the map $\Ccal \to \Xcal_i$ factors and is representable.

We call a twisted curve an \textbf{$r$-curve} (for some positive integer $r$) if the order of every stabiliser group divides $r$. Since a stable map $\Ccal_i \to \Xcal_i$ must be representable, it follows that $\Ccal_i$ is an $r_i$-curve.

A point of the fibre product $\mathcal P$ consists of the data of two stable maps $\Ccal_1 \to \Xcal_1$ and $\Ccal_2 \to \Xcal_2$ which induce the same underlying map $C \to X$ on coarse moduli.

\begin{lemma} Suppose that $r_1$ and $r_2$ are coprime, and let $\Ccal_1,\Ccal_2$ be $r_1$-,$r_2$-curves with the same coarse curve $C$. Then the normalisation
\begin{equation}\label{eqn: Ccal normalisation fibre product} \Ccal = \left(\Ccal_1 \times_C \Ccal_2\right)^{\sim} \end{equation}
is a twisted curve.
\end{lemma}
\begin{proof} If $p \in C$ is a marking on the coarse curve with local equation $z$, then the local model for each $\Ccal_i$ is given by:
\begin{equation*} \Ccal_i = \left[ (x_i^{r_i} = z)/\mu_{r_i}\right]. \end{equation*}
The fibre product is therefore $[(x_1^{r_1}=x_2^{r_2})/\mu_{r_1r_2}]$. Note that this is not a twisted curve (unless $r_1=1$ or $r_2=1$). On the other hand since $r_1$ and $r_2$ are coprime, the normalisation of the fibre product is given by
\[ [\mathbb A^1_y/\mu_{r_1r_2}]\]
where $y^{r_1}=x_2,y^{r_2}=x_1$. The computation around a node is entirely analogous except that the base must also be normalised, around the divisor where the node persists.
\end{proof}

\begin{remark}
This phenomenon is related to the issue of saturation in logarithmic geometry, via the correspondence between twisted curves and extensions of logarithmic structures \cite{Ols07}. Indeed, the monoid $\mathbb{N} e_1 \oplus_{r_1,\mathbb{N},r_2}\mathbb{N} e_2$ is not saturated: in the groupification $\mathbb{Z}^2/(r_1,-r_2)\cong\mathbb{Z}$, the image of the generator $e_1$ is divisible by $r_2$
\[ e_1=(a_1r_1+a_2r_2)e_1=a_1r_2e_2+a_2r_2e_1=r_2(a_1e_2+a_2e_1)\]
where $a_1,a_2 \in \mathbb{Z}$ are such that $a_1r_1+a_2r_2=1$. Similarly the image of $e_2$ is divisible by $r_1$.
\end{remark}

The twisted curve $\Ccal$ carries a natural map to $\Xcal$ which is clearly representable. We thus have a cartesian diagram
\begin{equation}\label{diagram}
\begin{tikzcd}
\Kup(\Xcal)\ar[r,"\nu"]\ar[d,"\varphi"]\ar[dr,phantom,"\Box"] & \Pcal \ar[d,"\psi"] \\
\Mfrak^{r_1r_2-\op{tw}}\ar[r] & \Mfrak^{r_1-\op{tw}}\times_{\Mfrak}\Mfrak^{r_2-\op{tw}}
\end{tikzcd}
\end{equation}
where the bottom morphism is the normalisation. The morphism $\varphi$ carries a natural perfect obstruction theory. We will now construct a compatible perfect obstruction theory for $\psi$. The diagram
\bcd
\Pcal\ar[r]\ar[d] \ar[rd,phantom,"\square"] & \Kup(\Xcal_1)\times_{\Mfrak}\Kup(\Xcal_2)\ar[d] \\
\Kup(X)\ar[r,"\Delta"] & \Kup(X)\times_{\Mfrak}\Kup(X)
\ecd
is cartesian. Using the convexity assumption, there is a perfect obstruction theory for $\Delta$ given by
\begin{equation}\label{eqn: pi f coarse} (\pi_{0\st}f_0^{\st}\Tan_X)^\vee[1] \end{equation}
where $\pi_0$ is the universal coarse curve. This pulls back to a perfect obstruction theory for $\Pcal \to \Kup(\Xcal_1)\times_{\Mfrak} \Kup(\Xcal_2)$. The latter space carries a perfect obstruction theory over $\Mfrak^{r_1-\op{tw}}\times_{\Mfrak}\Mfrak^{r_2-\op{tw}}$ given by:
\begin{equation}\label{eqn: pi f smooth root} (\pi_{1\st}f_1^{\st}\Tan_{\Xcal_1}\oplus\pi_{2\st}f_2^{\st}\Tan_{\Xcal_2})^\vee. \end{equation}
We thus have a triangle with perfect obstruction theories
\bcd
\Pcal \ar[r, "\eqref{eqn: pi f coarse}"] \ar[rr,bend right=20, "\psi"] & \Kup(\Xcal_1)\times_{\Mfrak} \Kup(\Xcal_2) \ar[r,"\eqref{eqn: pi f smooth root}"] & \Mfrak^{r_1-\op{tw}}\times_{\Mfrak}\Mfrak^{r_2-\op{tw}}
\ecd
and wish to build an obstruction theory for $\psi$ giving a compatible triple. There are natural morphisms $\Tan_{\Xcal_i} \to p_i^\star \Tan_X$ on $\Xcal_i$. We therefore obtain:
\begin{equation*} \pi_{1\st}f_1^{\st}\Tan_{\Xcal_1}\oplus\pi_{2\st}f_2^{\st}\Tan_{\Xcal_2} \to \pi_{0\st}f_0^{\st}\Tan_X.\end{equation*}
(As in the proof of Theorem~\ref{thm: local orbifold}, this follows from the projection formula and the fact that the structure sheaves of the various universal curves are preserved by pushforwards along coarsening maps, see \cite[Theorem 3.1]{AOV}.) Dualising, shifting and taking the cone, we obtain:
\begin{equation*} (\pi_{1\st}f_1^{\st}\Tan_{\Xcal_1}\oplus\pi_{2\st}f_2^{\st}\Tan_{\Xcal_2})^\vee \to \EE_\psi \to (\pi_{0\star} f_0^\star \Tan_X)^\vee[1] \xrightarrow{[1]}.\end{equation*}
Several applications of the Four Lemmas then show that $\EE_\psi$ is a relative perfect obstruction theory for $\psi$.

Finally, we wish to compare the obstruction theories of $\psi$ and $\varphi$ in \eqref{diagram}. For any root stack $\Ycal=Y_{D,r}$ with gerby divisor $\Dcal$, a local computation gives the following exact sequence:
\[0\to \Tan_{\Ycal}\to p^\star \Tan_Y\to \OO_{(r-1)\Dcal}(r\Dcal)\to 0.\]
From this we obtain a morphism of short exact sequences
\bcd
0\ar[r]& \Tan_{\Xcal}\ar[r]\ar[d] & p^\star \Tan_X\ar[r]\ar[d] & \bigoplus_{i=1}^2 \OO_{(r_i-1)\Dcal_i}(r_i\Dcal_i)\ar[d,equal]\ar[r]& 0\\
0\ar[r]& p_1^\star \Tan_{\Xcal_1}\oplus p_2^\star \Tan_{\Xcal_2}\ar[r]  & p^\star \Tan_X\oplus p^\star\Tan_X\ar[r] & \bigoplus_{i=1}^2 \OO_{(r_i-1)\Dcal_i}(r_i\Dcal_i)\ar[r]& 0
\ecd
and an application of the Snake Lemma produces the following exact sequence on $\Xcal$:
\begin{equation} 0\to \Tan_{\Xcal}\to  p_1^\star\Tan_{\Xcal_1}\oplus p_2^\star\Tan_{\Xcal_2}\to  p^\star\Tan_X\to 0 \end{equation}
Applying $\pi_\star f^\star$ we see that the pullback of the perfect obstruction theory for $\psi$ coincides with the perfect obstruction theory for $\varphi$ in \eqref{diagram}. The theorem then follows by the commutativity of virtual pullback and pushforward \cite[Theorem 4.1]{Mano12}, since the bottom horizontal arrow in \eqref{diagram} is proper of degree one.\end{proof}

\subsection{Local-orbifold correspondence} With the relative product formula established, we can now give a straightforward proof of Theorem~\ref{thm: local orbifold} in the convex setting.
\begin{proof}[Proof of Theorem~\ref{thm: local orbifold} for convex targets] Consider again the diagram \eqref{eqn: main cartesian square}. Theorem~\ref{thm: product formula} gives the following relation in $\Kup(X)$:
\begin{equation*} (\rho\circ\nu)_\star [\Kup(\Xcal)]^{\virt} = (\rho_1)_\star [\Kup(\Xcal_1)]^{\virt} \cdot (\rho_2)_\star [\Kup(\Xcal_2)]^{\virt}.\end{equation*}
Specialising to the maximal contact setting, the result immediately follows from the local-orbifold correspondence for smooth divisors and the splitting of the obstruction bundle for the local theory of $\OO_X(-D_1)\oplus \OO_X(-D_2)$.\end{proof}

The above result can be generalised to convex embeddings via virtual pullback methods. This is a fairly routine affair: see for instance \cite[Appendix A]{BN19}. Since the arguments in \S \ref{sec: rank reduction projection formula} already establish the result in full generality, we omit the details here.

\subsection{Comparison with naive invariants} \label{sec: naive} Recall from \cite{NabijouThesis,NR19} that for a simple normal crossings pair $(X,D)$ with $X$ convex, the naive virtual class is defined (in genus zero) as the product of logarithmic virtual classes
\begin{equation*} [\Nup(X|D)]^{\virt} := \prod_{i=1}^n (\rho_i)_\star [\Kup(X|D_i)]^{\virt}\end{equation*}
inside $\Kup(X)$. We also obtain a refined class on the fibre product $\Nup(X|D)$, but we are mostly interested in its pushforward to $\Kup(X)$. This definition extends to arbitrary convex embeddings via virtual pullback. An immediate consequence of Theorem~\ref{thm: product formula} is an identification of orbifold and naive invariants.

\begin{corollary} \label{cor: orbifold equals naive} For $(X,D)$ a convex embedding, the relation
\begin{equation*} \rho_\star[\Kup(X_{D,\vec{r}})]^{\virt} = [\Nup(X|D)]^{\virt}\end{equation*} 
holds inside $\Kup(X)$ (for compatible choices of contact orders).\end{corollary}

Given this, the (counter)examples presented in \cite[\S1]{NR19} and \cite[\S 3.4]{NabijouThesis} show that the orbifold invariants and logarithmic invariants differ in general, and that this defect is not restricted to the maximal contact setting.

The naive spaces provide an alternative perspective for probing the geography and invariants of the multi-root spaces. The iterated blowup construction of \cite{NR19} gives a method for comparing the logarithmic invariants to the naive/orbifold invariants; see also \cite{R19b,HerrProduct} for treatments of related ideas.

\footnotesize
\bibliographystyle{alpha}
\bibliography{Bibliography.bib}

\providecommand{\noopsort}[1]{}
\begin{thebibliography}{BBvG20b}

\bibitem[AC14]{AbramovichChenLog}
D.~Abramovich and Q.~Chen.
\newblock Stable logarithmic maps to {D}eligne-{F}altings pairs {II}.
\newblock {\em Asian J. Math.}, 18(3):465--488, 2014.

\bibitem[ACW17]{ACW}
D.~Abramovich, C.~Cadman, and J.~Wise.
\newblock Relative and orbifold {G}romov-{W}itten invariants.
\newblock {\em Algebr. Geom.}, 4(4):472--500, 2017.

\bibitem[AF16]{AF11}
D.~Abramovich and B.~Fantechi.
\newblock Orbifold techniques in degeneration formulas.
\newblock {\em Ann. Sc. Norm. Super. Pisa Cl. Sci. (5)}, 16(2):519--579, 2016.

\bibitem[AOV11]{AOV}
D.~Abramovich, M.~Olsson, and A.~Vistoli.
\newblock Twisted stable maps to tame {A}rtin stacks.
\newblock {\em J. Algebraic Geom.}, 20(3):399--477, 2011.

\bibitem[AV02]{AbramovichVistoli}
D.~Abramovich and A.~Vistoli.
\newblock Compactifying the space of stable maps.
\newblock {\em J. Amer. Math. Soc.}, 15(1):27--75, 2002.

\bibitem[BBvG19]{BBvG}
P.~Bousseau, A.~Brini, and M.~van Garrel.
\newblock {On the local-log principle for the toric boundary}.
\newblock {\em arXiv: 1908.04371}, 2019.

\bibitem[BBvG20a]{BBvG2}
P.~Bousseau, A.~Brini, and M.~van Garrel.
\newblock Stable maps to {L}ooijenga pairs.
\newblock {\em arXiv:2011.08830}, 2020.

\bibitem[BBvG20b]{BBvG3}
P.~Bousseau, A.~Brini, and M.~van Garrel.
\newblock {Stable maps to Looijenga pairs: orbifold examples}.
\newblock {\em arXiv:~2012.10353}, 2020.

\bibitem[BN21]{BN19}
L.~Battistella and N.~Nabijou.
\newblock Relative quasimaps and mirror formulae.
\newblock {\em Int. Math. Res. Not. IMRN}, (10):7885--7931, 2021.

\bibitem[CC08]{CadmanChen}
C.~Cadman and L.~Chen.
\newblock Enumeration of rational plane curves tangent to a smooth cubic.
\newblock {\em Adv. Math.}, 219(1):316--343, 2008.

\bibitem[CCIT09]{CCITTwisted}
T.~Coates, A.~Corti, H.~Iritani, and H.-H. Tseng.
\newblock Computing genus-zero twisted {G}romov-{W}itten invariants.
\newblock {\em Duke Math. J.}, 147(3):377--438, 2009.

\bibitem[FTY19]{FTYMirror}
H.~Fan, H.-H. Tseng, and F.~You.
\newblock Mirror theorems for root stacks and relative pairs.
\newblock {\em Selecta Math. (N.S.)}, 25(4):Paper No. 54, 25, 2019.

\bibitem[Ful98]{FultonIntersectionTheory}
W.~Fulton.
\newblock {\em Intersection theory}.
\newblock Ergebnisse der Mathematik und ihrer Grenzgebiete. Springer-Verlag,
  Berlin, second edition, 1998.

\bibitem[FW20]{FanWuWDVV}
H.~Fan and L.~Wu.
\newblock {Witten–Dijkgraaf–Verlinde–Verlinde Equation and its
  Application to Relative Gromov–Witten Theory}.
\newblock {\em Int. Math. Res. Not. IMRN}, online first, 2020.

\bibitem[GPS10]{GPS}
M.~Gross, R.~Pandharipande, and B.~Siebert.
\newblock The tropical vertex.
\newblock {\em Duke Mathematical Journal}, 153(2):297--362, 2010.

\bibitem[{Her}19]{HerrProduct}
L.~{Herr}.
\newblock {The Log Product Formula}.
\newblock {\em arXiv:~arXiv:1908.04936}, 2019.

\bibitem[JPT]{JPTNotes}
P.~Johnson, R.~Pandharipande, and H.-H. Tseng.
\newblock Notes on local $\mathbf{P}^1$ orbifolds.
\newblock \url{https://people.math.ethz.ch/~rahul/lPab.ps}.

\bibitem[KP08]{KlemmP08}
A.~Klemm and R.~Pandharipande.
\newblock {Enumerative geometry of Calabi-Yau 4-folds}.
\newblock {\em Comm. Math. Phys.}, 281(3):621--653, 2008.

\bibitem[Man12a]{Mano12}
C.~Manolache.
\newblock Virtual pull-backs.
\newblock {\em J. Algebr. Geom.}, 21(2):201--245, 2012.

\bibitem[Man12b]{ManolachePush}
C.~Manolache.
\newblock Virtual push-forwards.
\newblock {\em Geom. Topol.}, 16(4):2003--2036, 2012.

\bibitem[Nab18]{NabijouThesis}
N.~Nabijou.
\newblock Recursion formulae in logarithmic {G}romov--{W}itten theory and
  quasimap theory, 2018.
\newblock PhD~thesis, Imperial~College~London.
  \url{https://www.dpmms.cam.ac.uk/\~nn333/NabijouThesis.pdf}.

\bibitem[NR22]{NR19}
N.~Nabijou and D.~Ranganathan.
\newblock Gromov--{W}itten theory with maximal contacts.
\newblock {\em Forum Math. Sigma}, 10:Paper No. e5, 2022.

\bibitem[NS06]{NS06}
T.~Nishinou and B.~Siebert.
\newblock Toric degenerations of toric varieties and tropical curves.
\newblock {\em Duke Math. J.}, 135:1--51, 2006.

\bibitem[Ols07]{Ols07}
M.~Olsson.
\newblock ({L}og) twisted curves.
\newblock {\em Compos. Math.}, 143(02):476--494, 2007.

\bibitem[{Ran}19]{R19b}
D.~{Ranganathan}.
\newblock {A note on cycles of curves in a product of pairs}.
\newblock {\em arXiv:~1910.00239}, 2019.

\bibitem[TY20a]{TY20c}
H.-H. {Tseng} and F.~{You}.
\newblock {A Gromov-Witten theory for simple normal-crossing pairs without log
  geometry}.
\newblock {\em arXiv e-prints}, page arXiv:2008.04844, August 2020.

\bibitem[TY20b]{TY20b}
H.-H. {Tseng} and F.~{You}.
\newblock {A mirror theorem for multi-root stacks and applications}.
\newblock {\em arXiv:~2006.08991}, 2020.

\bibitem[TY20c]{TY18}
H.-H. Tseng and F.~You.
\newblock Higher genus relative and orbifold {G}romov--{W}itten invariants.
\newblock {\em Geom. Topol.}, 24(6):2749--2779, 2020.

\bibitem[vGGR19]{vGGR}
M.~van Garrel, T.~Graber, and H.~Ruddat.
\newblock {Local Gromov-Witten invariants are log invariants}.
\newblock {\em Adv. Math.}, 350:860--876, 2019.

\end{thebibliography}
\medskip

\textsc{Mathematisches Institut, Ruprecht-Karls-Universit\"at Heidelberg, Im Neuenheimer Feld 205, 69120 Heidelberg, Germany}

\indent \emph{Email address:} \href{lbattistella@mathi.uni-heidelberg.de}{lbattistella@mathi.uni-heidelberg.de} \medskip

\textsc{Department of Pure Mathematics and Mathematical Statistics, University of Cambridge, Centre for Mathematical Sciences, Wilberforce Road, Cambridge CB3 0WB, United Kingdom}

\indent \emph{Email address:} \href{nn333@cam.ac.uk}{nn333@cam.ac.uk} \medskip

\textsc{Department of Mathematics, Ohio State University, 100 Math Tower, 231 West 18th Ave., Columbus, OH 43210, USA}

\indent \emph{Email address:} \href{hhtseng@math.ohio-state.edu}{hhtseng@math.ohio-state.edu} \medskip

\textsc{Department of Mathematics, University of Oslo, Niels Henrik Abels hus, Moltke Moes vei 35, 0851 Oslo, Norway}

\indent \emph{Email address:} \href{youf@math.uio.no}{youf@math.uio.no}

\end{document}